\documentclass[11pt,bezier]{article}
\usepackage{amsmath,amssymb,amsfonts,euscript,mathtools}
\usepackage{amsthm}
\usepackage{enumerate}
\usepackage{cite}
\usepackage[bookmarks]{hyperref}
\usepackage{cleveref}

\usepackage{ifpdf}
\ifpdf    % we are running pdflatex
\usepackage{hyperref}
\else    % we are running latex
\usepackage[hypertex]{hyperref}
\fi

\newtheorem{prethm}{{\bf Theorem}}

\newenvironment{thm}{\begin{prethm}{\hspace{ - 0.5
em}}}{\end{prethm}}
\newtheorem{preremark}[prethm]{Remark}
\newenvironment{rem}{\begin{preremark}{\hspace{ - 0.5
em}}}{\end{preremark}}
\newcommand{\Gb}{{\overline G}}

\usepackage{tikz}
\usetikzlibrary{calc}

\tikzstyle{sensor}=[
draw,
fill=blue!20,
text width=2em,
text centered,
minimum height=2.5em
]
\tikzstyle{naveqs}=[
sensor,
text width=3em,
minimum height=6em,
]
\tikzstyle{naveqsT}=[
sensor,
text width=3.2em,
minimum height=7em,
dashed,
]
\tikzstyle{naveqsD}=[
sensor,
text width=3em,
minimum height=1em,
]
\tikzstyle{vertex}=[circle, draw, fill, inner sep=0pt, minimum size=3pt]
\newcommand{\vertex}{\node[vertex]}

\title{\large \bf Algebraic connectivity of the second power of a graph}

\author{{\normalsize
{\sc B. Afshari} 
\footnote{afshari.b@ipm.ir}
}
}
\date{}
\begin{document}

\maketitle

\begin{abstract}
Denote the Laplacian of a graph $G$ by $L(G)$ and its second smallest Laplacian eigenvalue by $\lambda_2(G)$. 
If $G$ is a graph on $n\ge 2$ vertices, then it is shown that the second smallest eigenvalue of $L(G) + \frac{1}{n} L(\overline{G^2})$ is at least 1, where $\overline{G^2}$ is the complement of the second power of $ G $. 
As a corollary of this result, it is shown that 
\begin{itemize}
	\item $ n \, \lambda_2(G) \ge \lambda_2(G^2), $
	\item $ \lambda_2(G) \ge 1-\frac{|D_G|}{n}, $
	\item $ \lambda_2(G) + \lambda_2(\Gb) \ge 1, $
\end{itemize}
where $|D_G|$ is the number of vertices of eccentricity at least 3 in $G$. 
\end{abstract}

%\vspace{3mm}
\noindent {\em 2010 {\rm AMS} Classification}: 05C50, 15A18

%15A42

%\section{Preliminaries and Results}
\hspace{4mm}

Let $ G $ be a simple graph with vertex set $ V = \{v_1, v_2, \ldots , v_n\} $ and edge set $E$. 
We denote the \textit{complement} of $ G $ by $ \Gb $. 
For $v_i$, denote the set of neighbors of $v_i$ in $G$ by $N_i$ and its number by $d_i$. 
The \textit{distance} between vertices $v_i$ and $ v_j $, denoted
by $ \mathrm{dist}_G(v_i,v_j) $, is the number of edges in a shortest path joining them. 
If there is no such path, then we define $ \mathrm{dist}_G(v_i,v_j)=\infty $. 
The \textit{second power} of $ G $, denoted by $ G^2 $, is the graph with the same vertex set as $ G $ such that two vertices are adjacent in $ G^2 $ if and only if their distance is at most $ 2 $ in $ G $. 
Also denote the set of all (unordered) pairs of vertices with distance 2 in $ G $ by $\pi(G)$, that is,  
$$\pi(G) = \big\{ \{v_i, v_j\} :\, \mathrm{dist}_G(v_i,v_j) = 2 \big\}.$$
The \textit{eccentricity} of a vertex in $G$ is defined as the length of a longest shortest path starting at that vertex. 
Denote by $ D_G $ the set of all vertices of eccentricity at least 3 in $G$.

Let $ A(G) $ be the adjacency matrix of $ G $ and $D(G) = \mathrm{diag}(d_{1},d_{2},\ldots,d_{n})$. 
The \textit{Laplacian} matrix of $ G $ is $L(G) = D(G)-A(G)$. 
Clearly, $ L(G) $ and $L(G) + \frac{1}{n} L(\overline{G^2})$ are real symmetric matrices. 
From this and Ger\v{s}gorin's Theorem, it follows that the eigenvalues of these matrices are nonnegative real numbers. 
Denote the eigenvalues of $L(G)$ by
$$0=\lambda_1(G) \le \lambda_2(G) \le \cdots \le \lambda_n(G),$$
and the eigenvalues of $L(G) + \frac{1}{n} L(\overline{G^2})$ by
$$0=\lambda'_1(G) \le \lambda'_2(G) \le \cdots \le \lambda'_n(G),$$
where the corresponding eigenvector for $\lambda_1(G)$ and $\lambda'_1(G)$ is the all-one vector. 
The value $\lambda_2(G)$ is called the \textit{algebraic connectivity} of $ G $. 
%In this paper, it is shown that
\begin{thm}\label{mainthm}
	For any graph $ G $ of order $ n \ge 2 $, 
	\begin{enumerate}[\rm (i)]
		\item $\lambda'_2(G) \ge 1$, or equivalently, for any vector $z$, 
%		$n \, z^T L(G) z \ge z^T L(G^2) z$,
		\begin{align*}
			n \, z^T L(G) z \ge z^T L(G^2) z,
		\end{align*}
		\item $n \, \lambda_{2}(G) \ge \lambda_{2}(G^2)$,
		\item $\lambda_2(G) \ge 1-\frac{|D_G|}{n}$, 
		\item $\lambda_2(G) + \lambda_2(\Gb) \ge 1$. 
	\end{enumerate} 
\end{thm}
\smallskip
\begin{rem} 
	\rm 
	The item (iv) of Theorem~\ref{mainthm} was first proposed as a conjecture in \cite{yl,zsh} and then was studied in multiple studies \cite{fxwl,btf,flt,lsl,lw,cw,l,xm,zsh,at,cd,bamm,ba1,ba2} and was recently proved in \cite{ek}. 
\end{rem}

%\smallskip
%\subsection{Proof}
\begin{proof}[Proof of Theorem \ref{mainthm}] 
(i)$\Rightarrow$(ii): 
		This is obvious. 

\smallskip
(ii)$\Rightarrow$(iii): 
It is easy to see that the algebraic connectivity of a graph of order $n$ is at least the number of vertices with degree $n-1$. 
Now, using $(ii)$ and noting that a vertex of eccentricity at most 2 in $G$ is adjacent to all vertices in $G^2$, we get 
\begin{align*}
	\lambda_2(G) 
	\ge \frac{1}{n} \lambda_2(G^2) 
	\ge 1-\frac{|D_G|}{n}.
\end{align*}
%and so, $(ii)$ implies $(iii)$. 

\smallskip
(iii)$\Rightarrow$(iv): 
Note that $D_G \cap D_{\Gb} = \emptyset$. 
To see this, let $v_i\in D_{\Gb}$. 
So, there is some $v_k$ such that $\mathrm{dist}_{\Gb}(v_i,v_k) \ge 3$. 
This is equivalent to say that $N_{i} \cup N_{k} = V$. 
So, $\mathrm{dist}_G(v_i,v_j) \le 2$ for any $v_j\in V$, that is, $v_i\notin D_G$.  
Now, using $(iii)$, we get 
\begin{align*}
	\lambda_2(G) + \lambda_2(\Gb) 
	\ge 2- \frac{|D_G|+|D_{\Gb}|}{n}
	\ge 1 .
\end{align*}
%and so, $(iii)$ implies $(iv)$. 

\medskip
(i): 
	Note that $\lambda'_2(G) \ge 1$ means that for any vector $z$ orthogonal to the all-one vector, 
\begin{align*}
	\sum_{ \{v_r, v_s\} \in E }{ (z_r-z_s)^2 } + \frac{1}{n} \sum_{ \{v_r, v_s\} :\, \mathrm{dist}_G(v_r,v_s)\ge 3}{ (z_r-z_s)^2 } 
	\ge \sum_{v_r\in V}{ z_r^2 }. 
\end{align*}
	This is equivalent to saying that for any vector $z$, 
	\begin{equation}\label{eq1}
		\sum_{ \{v_r, v_s\} \in E }{ (z_r-z_s)^2 } + \frac{1}{n} \sum_{ \{v_r, v_s\} :\, \mathrm{dist}_G(v_r,v_s)\ge 3}{ (z_r-z_s)^2 } 
		\ge \sum_{v_r\in V}{ ( z_r- \frac{\sum_{v_r}{z_r}}{n} )^2 }. 
	\end{equation}
This is because the inequality~(\ref{eq1}) is invariant
under translating $z$ and we can assume that $\sum_{v_r}{z_r} = 0$.
	Now since
	$$\sum_{v_r \in V}{ ( z_r- \frac{\sum_{v_r}{z_r}}{n} )^2 } = \frac{1}{n} \sum_{ \{v_r, v_s\} \in E }{ (z_r-z_s)^2 },$$ 
the inequality~(\ref{eq1}) is equivalent to
	\begin{align*}
		(n-1) \sum_{ \{v_r, v_s\} \in E }{ (z_r-z_s)^2 } \ge \sum_{ \{v_r, v_s\} \in \pi(G)}{ (z_r-z_s)^2 }
	\end{align*}
that gives $n \, z^T L(G) z \ge z^T L(G^2) z$. 
	
	Our proof for $\lambda'_2(G)\ge1$ is by contradiction. 
Let $G$ be a graph with minimum number of vertices such that $\lambda'_2(G)<1$ and $x = (x_1, x_2, \ldots , x_n)^T$ be a unit eigenvector corresponding to $\lambda'_2(G)$. 
It follows that
	\begin{equation}\label{ContraryAssump}
		(n-1) \sum_{ \{v_r, v_s\} \in E }{ (x_r-x_s)^2 } < \sum_{ \{v_r, v_s\} \in \pi(G) }{ (x_r-x_s)^2 }. 
	\end{equation}
	Since the all-one vector is the eigenvector corresponding to $\lambda'_1(G)$, we have $\sum_{r=1}^n x_r = 0$. 
	Note that the claim holds if $n\le2$. 
	So, $n\ge3$. 

For $v_k\in V$, letting $p_k = \sum_{ v_r \in N_{k} }{ (x_k-x_r) }$ we get 
\begin{align*}
	\sum_{ \{v_r, v_s\} \subseteq N_{k} } {(x_r-x_s)^2 }
	&= d_{k} \sum_{ v_r\in N_{k} } { \big(x_r- (\frac{\sum_{v_s\in N_{k}} x_s}{d_{k}}) \big)^2 } \\
	&= \frac{1}{d_{k}} \sum_{ v_r\in N_{k} } { \big( d_{k} (x_r-x_k) + d_{k} x_k - \sum_{v_s\in N_{k}} x_s \big)^2 }  \\
	&= \frac{1}{d_{k}} \sum_{ v_r\in N_{k} } { \big( d_{k} (x_r-x_k) + p_k \big)^2 }  \\
	&= d_{k} \sum_{ v_r\in N_{k} } {(x_r-x_k)^2} \,+ p_k^2 + \sum_{ v_r\in N_k } {2\, p_k (x_r-x_k)}  \\
	&= d_{k} \sum_{ v_r\in N_{k} } {(x_r-x_k)^2}  \,- p_k^2 .
\end{align*}

For ease of notation let $\ell_{rs} = |x_r - x_s|$, $r,s=1,\ldots,n$. 
Now we can deduce that 
\begin{align*}
	\sum_{ \{v_r, v_s\} } |N_{r} \cap N_{s}| \ell_{rs}^2 
	&= \sum_{k=1}^{n} { \sum_{ \{v_r, v_s\} \subseteq N_{k} } \ell_{rs}^2 } \\
	&\le \sum_{k=1}^{n} { \big( d_{k} \sum_{ v_r\in N_{k} } {\ell_{rk}^2} \big)  } \\
	&= \sum_{ \{v_r, v_s\} \in E }  (d_{r}+d_{s}) \ell_{rs}^2,
\end{align*}
and in consequence, 
\begin{align*}
	\sum_{ \{v_r, v_s\} \in E }{(n-1) \ell_{rs}^2} - \sum_{ \{v_r, v_s\} \in \pi(G)}{\ell_{rs}^2} 
	&\ge \sum_{ \{v_r, v_s\} \in E }{(n-1) \ell_{rs}^2} - \sum_{ \{v_r, v_s\} \notin E}{|N_{r}\cap N_{s}|\ell_{rs}^2} \\
	&=  \sum_{ \{v_r, v_s\} \in E }{(n-1+|N_{r}\cap N_{s}|) \ell_{rs}^2} - \sum_{ \{v_r, v_s\} }{|N_{r}\cap N_{s}|\ell_{rs}^2} \\
	&\ge \sum_{ \{v_r, v_s\} \in E }{ (n-1-|N_{r}\cup N_{s}| ) \ell_{rs}^2} .
\end{align*} 
If $\ell_{rs}=0$ for any $\{v_r, v_s\} \in E$ with $|N_{r}\cup N_{s}| = n$, then the last term in the above is nonnegative, a contradiction. 
So there is some $\{v_i, v_j\} \in E$ with $N_{i}\cup N_{j} = V$ and $x_j<x_i$. 
Let $X = N_{i} \setminus (N_{j} \cup \{v_j\})$, $Y = N_{j} \setminus (N_{i} \cup \{v_i\})$, and $ Z= N_{i} \cap N_{j}$. 
Since (\ref{ContraryAssump}) is invariant under negating $x$, we may assume that $|Y|\ge |X|$. 
If $|X|=0$ then $d_{j}=n-1$ and so, 
\begin{align*}
	\lambda'_2(G)
	\ge \sum_{ v_r \in V } (x_r-x_j)^2
	\ge \sum_{ v_r \in V } x_r^2
	= 1,
\end{align*}
a contradiction. 
So, $|X|\ge 1$. 
Let $G_1$ be the graph obtained from $G$ by removing the edges with both ends in $V\backslash \{v_i,v_j\}$ (see Figure~\ref{fig:G1}).

\bigskip
	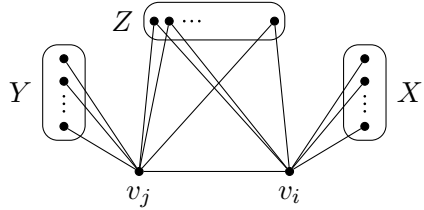
\begin{figure}[h]
	\centering
	\begin{tikzpicture}
		\vertex (y) at (-1,-.5) [label=below:$v_j$] {};
		\vertex (x) at (1,-.5) [label=below:$v_i$] {};
		
		\node (X) at (-2,.55) [draw,rounded corners=0.2cm,text width=.8em, minimum height=3.3em, label= left:{$ Y $}] {};
		\node (Y) at (2,.55) [draw,rounded corners=0.2cm,text width=.8em, minimum height=3.3em, label=right:{$ X $}] {};
		\node (Z) at (0,1.5) [draw,rounded corners=0.2cm,text width=4.2em, minimum height=1.3em, label=left:{$Z$}] {};
		%			\node (W) at (0,-1.5) [draw,rounded corners=0.2cm,text width=6em, minimum height=1.5em, label=left:{$W$}] {};
		% Y vertices:
		\vertex (y1) at (-2,1) {};
		\vertex (y2) at (-2,.7) {};
		\fill ($(y2)+(0,-0.2)$) circle[radius=.5pt] +(0,-0.1) circle[radius=.5pt]  +(0,-0.2) circle[radius=.5pt] ;
		\vertex (y3) at (-2,.1) {};
		% X vertices:
		\vertex (x1) at (2,1) {};
		\vertex (x2) at (2,.7) {};
		\fill ($(x2)+(0,-0.2)$) circle[radius=.5pt] +(0,-0.1) circle[radius=.5pt]  +(0,-0.2) circle[radius=.5pt] ;
		\vertex (x3) at (2,.1) {};
		% Z vertices:
		\vertex (z1) at (-.8,1.5) {};
		\vertex (z2) at (-.6,1.5) {};
		\fill ($(z2)+(0.2,0)$) circle[radius=.5pt] +(0.1,0) circle[radius=.5pt]  +(0.2,0) circle[radius=.5pt] ;
		\vertex (z3) at (.8,1.5) {};
		\path
		(x) edge (y)
		(x) edge (x1)
		(x) edge (x2)
		(x) edge (x3)
		(x) edge (z1)
		(x) edge (z2)
		(x) edge (z3)
		(y) edge (y1)
		(y) edge (y2)
		(y) edge (y3)
		(y) edge (z1)
		(y) edge (z2)
		(y) edge (z3)	;
	\end{tikzpicture}
%	\caption{The Graph $G$ with sets $X$, $Y$ and $Z$ specified (note that only edges incident in $\{v_i,v_j\}$ are shown)}
	\caption{The Graph $G_1$}
	\label{fig:G1}
\end{figure}

Now we show that \begin{equation}\label{g1}
	\sum_{ \{v_r, v_s\} \in E(G_1) }{ (n-1) \ell_{rs}^2 } \ge \sum_{ \{v_r, v_s\} \in \pi(G_1) }{ \ell_{rs}^2 } .
\end{equation} 
Note that 
\begin{align*}
	\sum_{ \{v_r, v_s\} \in \pi(G_1) }{ \ell_{rs}^2 }
	\le \sum_{ \{v_r, v_s\} \subseteq N_{i} }{ \ell_{rs}^2 }  + \sum_{ \{v_r, v_s\} \subseteq N_{j} }{ \ell_{rs}^2 } - \sum_{ v_r \in Z }{ (\ell_{ir}^2 + \ell_{jr}^2) } . 
\end{align*}
So,  
\begin{align*}
	&\sum_{ \{v_r, v_s\} \in E(G_1) }{ (n-1) \ell_{rs}^2 } \\
	&= \sum_{ v_r \in N_{i} }{ d_{i} \ell_{ir}^2 } + \sum_{ v_r \in N_{j} }{ d_{j} \ell_{jr}^2 }  + \sum_{ v_r \in X\cup Z }{ |Y| \ell_{ir}^2 } 
	+  \sum_{ v_r \in Y\cup Z }{ |X| \ell_{jr}^2 } 
	\,- (|Z|+1) \ell_{ij}^2 \\
	&= p_i^2 + \sum_{ \{v_r, v_s\} \subseteq N_{i} }{ \ell_{rs}^2 }  \,+ p_j^2 + \sum_{ \{v_r, v_s\} \subseteq N_{j} }{ \ell_{rs}^2 } + \sum_{ v_r \in X\cup Z }{ |Y| \ell_{ri}^2 } 
+ \sum_{ v_r \in Y\cup Z }{ |X| \ell_{rj}^2 } - (|Z|+1) \ell_{ij}^2  \\
		&\ge \sum_{ \{v_r, v_s\} \in \pi(G_1) }{ \ell_{rs}^2 } \,+ p_i^2 + p_j^2 + \sum_{ v_r \in X }{ |Y| \ell_{ir}^2 } 
		+ \sum_{ v_r \in Y }{ |X| \ell_{jr}^2 } + \sum_{ v_r \in Z }{ 2(\ell_{ir}^2+ \ell_{jr}^2 ) } \,- (|Z|+1) \ell_{ij}^2 . 
\end{align*} 
If there is some $v_t\in Z$ such that $x_t\ge x_i$ or $x_t\le x_j$, then 
\begin{align*}
	 \sum_{ v_r \in Z }{ 2(\ell_{ir}^2+ \ell_{jr}^2 ) } 
	 \ge 2 \ell_{ij}^2 +  \sum_{ v_r \in Z\backslash\{v_t\} }{ 2 (\ell_{ir}^2+ \ell_{jr}^2 ) }
	 \ge (|Z|+1) \ell_{ij}^2,
\end{align*}
and the inequality~(\ref{g1}) follows. 
So assume that $x_j < x_r < x_i$ for any $v_r\in Z$. 

Note that if $k,\ell$ are positive integers and $z, y_1,y_2,\ldots,y_k$ are $k+1$ values such that $z + \sum_{r=1}^{k}{y_r} = s$ for some fixed $s$, then applying the Cauchy-Schwarz inequality on the vectors $(z,\sqrt{\ell} y_1,\sqrt{\ell} y_2,\ldots,\sqrt{\ell} y_k)^T$ and $(1,\frac{1}{\sqrt{\ell}},\frac{1}{\sqrt{\ell}},\ldots,\frac{1}{\sqrt{\ell}})^T$ yields
\begin{equation*}%\label{ineq_minimizingSumOfSquares}
	z^2 + \ell\sum_{r=1}^{k}{y_r^2} 
	\ge \frac{\ell s^2}{k+\ell} .
\end{equation*}

In consequence, if $a = p_i + \sum_{ v_r \in X }{ (x_r-x_i) }$ and $b = -p_j - \sum_{ v_r \in Y }{ (x_r-x_j) }$ then
\begin{align*}
	p_i^2 + |Y| \sum_{ v_r \in X }{ (x_r-x_i)^2 } 
	\ge \frac{|Y| a^2}{|X|+|Y|}, 
	\end{align*}
	and 
	\begin{align*}
	p_j^2 + |X| \sum_{ v_r \in Y }{ (x_r-x_j)^2 } 
	\ge \frac{|X| b^2}{|X|+|Y|}.
\end{align*} 

Note that 
$$a = p_i + \sum_{ v_r \in X }{ (x_r-x_i) } 
= (x_i - x_j) + \sum_{ v_r \in Z } (x_i-x_r)$$
and 
$$b = -p_j - \sum_{ v_r \in Y }{ (x_r-x_j) } 
= (x_i - x_j)+ \sum_{ v_r \in Z } (x_r-x_j).$$
%It follows that $a+b = (|Z|+2) (x_i-x_j)$, and s
Since $x_j < x_r < x_i$ for any $v_r\in Z$, we have $a\ge (x_i-x_j)$ and $b\ge (x_i-x_j)$. 
Now we can deduce that
\begin{align*}
	p_i^2 + p_j^2 + \sum_{ v_r \in X }{ |Y| \ell_{ir}^2 } 
	+ \sum_{ v_r \in Y }{ |X| \ell_{jr}^2 } + \sum_{ v_r \in Z }{ 2(\ell_{ir}^2+ \ell_{jr}^2 ) }
	&\ge  \frac{|Y| a^2 + |X| b^2}{|X|+|Y|} + |Z| \ell_{ij}^2 \\
		&\ge (|Z|+1) \ell_{ij}^2,
\end{align*} 
and the inequality~(\ref{g1}) follows. 

Let $G_2$ be the graph obtained from $G$ by removing the vertices $v_i$ and $v_j$. 
By minimality of $G$, we have 
\begin{align*}
	(n-3)\, \sum_{ \{v_r, v_s\} \in E(G_2) }{ \ell_{rs}^2 } 
	\ge \sum_{ \{v_r, v_s\} \in \pi(G_2) }{ \ell_{rs}^2 } . 
\end{align*} 
It is easy to see that if $\{v_r, v_s\}\in \pi(G)$ for some $v_r\in X, v_s\in Y$, then $\{v_r, v_s\}\in \pi(G_2)$. So, 
\begin{align*}
	\sum_{ \{v_r, v_s\} \in \pi(G_2) }{ \ell_{rs}^2 } 
	\ge \sum_{ v_r\in X, v_s\in Y:\, \{v_r, v_s\} \in \pi(G) }{ \ell_{rs}^2 } . 
\end{align*} 

Now we can deduce that 
\begin{align*}
	(n-1)\, \sum_{ \{v_r, v_s\} \in E }{ \ell_{rs}^2 } 
	&= (n-1) \sum_{ \{v_r, v_s\} \in E(G_1) }{ \ell_{rs}^2 } \,+ (n-1)\sum_{ \{v_r, v_s\} \in E(G_2) }{ \ell_{rs}^2 } \\
	&\ge \sum_{ \{v_r, v_s\} \in \pi(G_1) }{ \ell_{rs}^2 } + \sum_{ v_r\in X, v_s\in Y:\, \{v_r, v_s\} \in \pi(G) }{ \ell_{rs}^2 } \\
	&\ge \sum_{ \{v_r, v_s\} \in \pi(G) }{ \ell_{rs}^2 }, 
\end{align*} 
a contradiction. 
Now we are done. 
\end{proof}

\bigskip

\noindent \textbf{Acknowledgement.} We would like to thank the anonymous referees for valuable comments, corrections and suggestions, which resulted in an improvement of the original manuscript. 

\bigskip

\bibliographystyle{amsplain}

\begin{thebibliography}{5}

\bibitem{ba1}  B. Afshari and S. Akbari, {\em A note on the algebraic connectivity of a graph and its complement}, {Linear Multilinear A.} \textbf{69} (2019) 1248--1254. 

\bibitem{ba2} B. Afshari and S. Akbari, {\em Some results on the Laplacian spread conjecture}, {Linear Algebra Appl.} \textbf{574} (2019) 22--29. 

\bibitem{bamm}  B. Afshari et al., {\em The algebraic connectivity of a graph and its complement}, {Linear Algebra Appl.} \textbf{555} (2018) 157--162. 

\bibitem{at} F. Ashraf and B. Tayfeh-Rezaie, {\em Nordhaus--Gaddum type inequalities for Laplacian and signless Laplacian eigenvalues}, {Electron. J. Combin.} \textbf{21} (2014), Paper 3.6, 13 pp.

\bibitem{btf} Y.-H. Bao, Y.-Y. Tan, and Y.-Z. Fan, {\em The Laplacian spread of unicyclic graphs}, {Appl. Math. Lett.} \textbf{22} (2009)  1011--1015.

\bibitem{cd} X. Chen and K.C. Das, {\em Some results on the Laplacian spread of a graph}, {Linear Algebra Appl.} \textbf{505} (2016) 245--260.

\bibitem{cw} Y. Chen and L. Wang, {\em The Laplacian spread of tricyclic graphs}, {Electron. J. Combin.} \textbf{16} (2009) Research Paper 80, 18 pp. 

\bibitem{ek} M. Einollahzadeh and M. M. Karkhaneei, {\em On the lower bound of the sum of the algebraic connectivity of a graph and its complement}, {J. Combin. Theory Ser. B} \textbf{151} (2021) 235--249.


\bibitem{flt} Y.Z. Fan, S.D. Li, and Y.Y. Tan, {\em The Laplacian spread of bicyclic graphs}, {J. Math. Res. Exposition} \textbf{30} (2010) 17--28.

\bibitem{fxwl} Y.-Z. Fan et al., {\em The Laplacian spread of a tree}, {Discrete Math. Theor. Comput. Sci.} \textbf{10} (2008) 79--86.

\bibitem{lsl} P. Li, J.S. Shi, and R.L. Li, {\em Laplacian spread of bicyclic graphs}, (Chinese) {J. East China Norm. Univ. Natur. Sci. Ed.} \textbf{1} (2010) 6--9.

\bibitem{l} Y. Liu, {\em The Laplacian spread of cactuses}, {Discrete Math. Theor. Comput. Sci.} \textbf{12} (2010) 35--40.

\bibitem{lw} Y. Liu and L. Wang, {\em The Laplacian spread of bicyclic graphs}, {Advances in Mathematics\,(China)} \textbf{40} (2011) 759--764.

\bibitem{xm} Y. Xu and J. Meng, {\em The Laplacian spread of quasi-tree graphs}, {Linear Algebra Appl.} \textbf{435} (2011) 60--66.

\bibitem{yl} Z. You and B. Liu, {\em The Laplacian spread of graphs}, {Czechoslovak Math. J.} \textbf{62} (137) (2012) 155--168.

\bibitem{zsh} M. Zhai, J. Shu, and Y. Hong, {\em On the Laplacian spread of graphs}, {Appl. Math. Lett.} \textbf{24} (2011) 2097--2101.

\end{thebibliography}
%%%%%%%%%%%%%%%%%%%%%%%%%%%%%%%%%%%%%%%%%%%
%

\end{document}